\documentclass[11pt,a4paper,reqno]{amsart}

\usepackage{amsmath,amssymb,amsthm,amscd,mathrsfs}

\usepackage[T1]{fontenc}
\usepackage[latin1]{inputenc}
\usepackage[english]{babel}
\usepackage[mathscr]{eucal}
\usepackage{graphicx}

\input{comment.sty}
\includecomment{FZ}
\includecomment{MM}
\numberwithin{equation}{section}

\newtheorem{theorem}{Theorem}[section]
\newtheorem{proposition}[theorem]{Proposition}
\newtheorem{lemma}[theorem]{Lemma}

\theoremstyle{definition}
\newtheorem{definition}[theorem]{Definition}

\theoremstyle{remark}
\newtheorem{remark}[theorem]{Remark}

\begin{document}

\title[ ]{Nonlinear parabolic equations with soft measure data}
\author{M. ABDELLAOUI AND E. AZROUL}
\address[MOHAMMED. ABDELLAOUI]{University of Fez, Faculty of Sciences Dhar El Mahraz, Laboratory LAMA, Department of Mathematics, B.P. 1796, Atlas Fez, Morocco.}
\email{mohammed.abdellaoui3@usmba.ac.ma}

\address[ELHOUSSINE. AZROUL]{University of Fez, Faculty of Sciences Dhar El Mahraz, Laboratory LAMA, Department of Mathematics, B.P. 1796, Atlas Fez, Morocco.}
\email{elhoussine.azroul@usmba.ac.ma}
\thanks{}

%
\subjclass{65J15, 28A12, 35B45, 35A35, 35Q35}
\keywords{Porous media equation, parabolic $p-$capacity, renormalized solution, a priori-estimates, equidiffuse measure}
\date{July 19, 2018.} 

\begin{abstract} 
In this paper we prove existence and uniqueness results for nonlinear parabolic problems with Dirichlet boundary values whose model is  

\[
\left\{
\begin{aligned}
&b(u)_t-\Delta_{p}u=\mu\;\mbox{in }(0,T)\times\Omega,\\
&b(u(0,x))=b(u_{0})\;\mbox{in }\Omega,\\
&u(t,x)=0\;\mbox{on }(0,T)\times\partial\Omega.
\end{aligned}
\right.
\]
where $\Delta_{p}u=\text{div}(|\nabla u|^{p-2}\nabla u)$ is the usual $p-$Laplace operator, $b$ is a increasing $C^{1}-$function and $\mu$ is a finite measure which does not charge sets of zero parabolic $p-$capacity, and we discuss their main properties.
\end{abstract} 

\maketitle  
\section{Introduction}\label{Sect. 1}   
Let $\Omega$ be an open bounded subset of $\mathbb{R}^{N}$ $(N\geq 2)$, $T$ is a positive real number, $p>1$, and let us consider the model problem
\begin{equation}\label{Eq 1.1} 
\left\{ 
\begin{aligned}   
&\frac{\partial b(u)}{\partial t}-\Delta_{p}u=\mu \quad\mbox{in }(0,T)\times\Omega\\   
&b(u)=b(u_{0}) \quad\mbox{on }\lbrace 0\rbrace\times\Omega\\   
&u=0 \quad\mbox{on }(0,T)\times\partial\Omega,   
\end{aligned} 
\right.  
\end{equation} 
where $u_{0}$ is a measurable function such that $b(u_{0})\in L^{1}(\Omega)$ and $\mu$ is a bounded Radon measure on $Q=(0,T)\times\Omega$.\\
It is well known that, if $b(u)=u$, $\mu\in L^{p'}(Q)$ and $u_{0}\in L^{2}(\Omega)$, J.-L. Lions \cite{L} proved existence and uniqueness of a weak solution. Under the general assumptions that $\mu$ and $u_{0}$ are bounded measures, the existence of a distributional solution was proved in \cite{BDGO}, by approximating $(\ref{Eq 1.1})$ with problems having regular data and using compactness arguments, due to the lack of regularity of the solutions, the distributional formulation is not strong enought to provide uniqueness, as it can be proved by adapting the conterexemple of J. Serrin to the parabolic case. However, for nonlinear operators with $L^{1}-$data, a new concept of solutions was done in \cite{BM} and in \cite{Pr} (see also \cite{B6}), where the notions of renormalized solution, and entropy solution, respectively, were introduced. If $\mu$ is a measure that does not charge sets of zero parabolic $p-$capacity (the so called diffuse measures), the notion of renormalized solution was introduced in \cite{DPP}. In \cite{DP} a similar notion of entropy solution is also defined, and proved to be equivalent to the renormalized one. The case in which $b$ is a strictly increasing $C^{1}-$function and $\Delta_{p}$ is a $p-$Laplace operator (i.e. $(\ref{Eq 1.1})$) was faced in \cite{BPR} if $\mu$ is a diffuse measure (see also \cite{Pe1} when $\mu$ is general). All these latest results are strongly based on a decomposition theorem given in \cite{DPP}, the key point in the existence result being the proof of the strong compactness of suitable truncations of the approximating solutions in the energy space.\\
Recently, in \cite{PPP1} (see also \cite{PPP2}) the authors proposed a new approach to the same problem with diffuse measures as data. This approach avoids to use the particular structure of the decomposition of the measure and it seems more flexible to handle a fairly general class of problems. In order to do that, the authors introduced a definition of renormalized solution which is closer to the one used for conservation laws in \cite{BCW} and to one of the existing formulations in the elliptic case (see \cite{DM} and \cite{DMOP}). Our goal is to extend the approach in \cite{PPP2} to the framework of the so-called generalized porous medium equation of the type $v_{t}-\Delta_{p}\psi(v)$ with $\psi(v)=u$ and $\psi^{-1}=b$, $\psi$ is a strictly increasing function.\\

The paper is organized as follows. In sect. \ref{Sect. 2} we give some preliminaries on the notion of parabolic $p-$capacity and on the functional spaces and some basic notations and properties. Sect. \ref{Sect. 3} is devoted to set the main assumptions and the new renormalized formulation of problem $(\ref{Eq 1.1})$. In sect. \ref{Sect. 4}, we prove that the definition of renormalized solution does not depend on the classical decomposition of $\mu$. In sect. \ref{Sect. 5} we give the proof of the main result (Theorem 5.1). We will briefly sketch in Sect. \ref{Sect. 6} the proof of the uniqueness of the solution.

\section{Preliminaries on parabolic capacity}\label{Sect. 2} 
Given a bounded open set $\Omega\subset\mathbb{R}^{N}$ and $T>0$, let $Q=(0,T)\times\Omega$. We recall that for every $p>1$ and every open subset $U\subset Q$, the parabolic $p-$capacity of $U$  (see \cite{DPP},\cite{P} and \cite{PPP1}) is given by
\begin{equation}\label{Eq 2.1}
\text{cap}_{p}(U)=\text{inf}\lbrace \Vert u\Vert_{W}:u\in W, u\geq\chi_{U}\text{ a.e. in }Q\rbrace,
\end{equation}
where 
\begin{equation}\label{Eq 2.2}
W=\lbrace u\in L^{p}(0,T;V): u_{t}\in L^{p'}(0,T;V')\rbrace
\end{equation}
Let us recall that $V=W^{1,p}_{0}(\Omega)\cap L^{2}(\Omega)$ endowed with its natural norm $\Vert\cdot\Vert_{W^{1,p}_{0}(\Omega)}+\Vert\cdot\Vert_{L^{2}(\Omega)}$ and $V'$ is its dual space. As usual $W$ is endowed with the norm
\[ \Vert u\Vert_{W}=\Vert u\Vert_{L^{p}(0,T;V)}+\Vert u_{t}\Vert_{L^{p'}(0,T;V')}\]
As usual we set $\text{inf }\emptyset=+\infty$. The parabolic capacity $\text{cap}_{p}$ is then extended to arbitrary Borel subset $B$ of $Q$ as
\[ \text{cap}_{p}(B)=\text{inf}\lbrace \text{cap}_{p}(U): B\subset U\text{ and }U\subset Q\text{ is open}\rbrace.\]
We denote by $\mathcal{M}_{b}(Q)$ the set of all Radon measures with bounded variation on $Q$ equipped with the norm $\Vert \mu\Vert_{\mathcal{M}_{b}(Q)}=|\mu|(Q)$.\\
We call a measure $\mu$ diffuse if $\mu(E)=0$ for every Borel set $E\subset Q$ such that $\text{cap}_{p}(E)=0$, $\mathcal{M}_{0}(Q)$ will denote the subspace of all diffuse measures in $Q$.\\
Difuse measures play an important role in the study of boundary value problems with measures as source terms. Indeed, for such measures one expects to obtain conterparts (in some generalized framework) of existence and uniqueness results known in the variational setting. Properties of diffuse measures in connection with the resolution of nonlinear parabolic problems have been investigated in \cite{DPP}. In that paper, the authors proved that for every $\mu\in\mathcal{M}_{0}(Q)$, there exists $f\in L^{1}(Q)$, $g\in L^{p}(0,T;V)$ and $\chi\in L^{p'}(0,T;W^{-1,p'}(\Omega))$ such that 
\begin{equation}\label{Eq 2.3}
 \mu=f+g+\chi\text{   in   }\mathcal{D}'(Q)
\end{equation}
Note that the decomposition in $(\ref{Eq 2.3})$ is not uniquely determined and the presence of the term $g_{t}$ is essentially due to the presence of diffuse measures which charges sections of the parabolic cylinder $Q$ and gives some extra difficulties in the study of this type of problems; in particular the parabolic case with absorption term $h(u)$. The main reason is that a solution of 
\[ u_{t}-\Delta_{p}u+h(u)=\mu=f+\chi+g\text{   in   }Q\] is meant in the sens that $v=u-g$ satisfies
\[ v_{t}-\Delta_{p}(v+g)+h(v+g)=f+\chi\text{   in   }Q.\]
However, since no growth restriction is made on $h$, the proof is a hard technical issue if $g$ is not bounded. For further considerations on this fact we refer to \cite{BP} (see also \cite{BMP}, \cite{PPP1}) and references therein.\\
In \cite{PPP1}, the authors also proved the following approximation theorem for an arbitrary diffuse measure that is essentially independent on the decomposition of the measure data.
\begin{theorem}
Let $\mu\in\mathcal{M}_{0}(Q)$. Then, for every $\epsilon>0$, there exists $\nu\in\mathcal{M}_{0}(Q)$ such that 
\begin{equation}\label{Eq 2.4}
\Vert \mu-\nu\Vert_{\mathcal{M}(Q)}\leq\epsilon \text{   and   }\nu=w_{t}-\Delta_{p}w\text{   in   }\mathcal{D}'(Q) 
\end{equation}
where $w\in L^{p}(0,T;W^{1,p}_{0}(\Omega))\cap L^{\infty}(Q)$.
\end{theorem}
Note that the function $w$ is constructed as the truncation of a nonlinear potential of $\mu$.\\

We will argue by density for proving the existence of a solution, so that we need the following preliminary result whose proof can be found, for instance, in \cite{PPP1} (see also Appendix).
\begin{proposition}\label{Prop 2.2}
Given $\mu\in\mathcal{M}(Q)\cap L^{p'}(0,T;W^{-1,p'}(\Omega))$ and $u_{0}\in L^{2}(\Omega)$, let $u\in W$ be the (unique) weak solution of 
\begin{equation}\label{Eq 2.5}
\left\{  
\begin{aligned}   
&u_{t}-\Delta_{p}u=\mu \quad\mbox{in }Q\\   
&u=u_{0} \quad\mbox{on }\lbrace 0\rbrace\times\Omega\\   
&u=0 \quad\mbox{on }(0,T)\times\partial\Omega,   
\end{aligned} 
\right. 
\end{equation}
Then,
\begin{equation}\label{Eq 2.6}
 \text{cap}_{p}(\lbrace |u|>k\rbrace)\leq C\text{max}\lbrace \frac{1}{k^{\frac{1}{p}}},\frac{1}{k^{\frac{1}{p'}}}\rbrace\quad\forall k\geq 1,
 \end{equation}
where $C>0$ is a constant depending on $\Vert \mu\Vert_{\mathcal{M}(Q)}$, $\Vert u_{0}\Vert_{L^{1}(\Omega)}$, and $p$.
\end{proposition}
Note that the proof of the corresponding Proposition in our case is postponed to the Appendix in Sect. \ref{Sect. 7}.
\begin{definition}\label{Def 2.3}
A sequence of measures $(\mu_{n})$ in $Q$ is equidiffuse, if for every $\eta>0$ there exists $\delta>0$ such that
\[ \text{cap}_{p}(E)<\delta\Longrightarrow|\mu_{n}|(E)<\eta\quad\forall n\geq 1.\]
The following result is proved in \cite{PPP2}.
\begin{lemma}\label{Lem 2.4}
Let $\rho_{n}$ be a sequence of mollifiers on $Q$. If $\mu\in\mathcal{M}_{0}(Q)$, then the sequence $(\rho_{n}*\mu_{n})$ is equidiffuse.
\end{lemma}
\end{definition}
\begin{flushleft}
Here are some notations we will use throughout the paper.
\end{flushleft}

We consider a sequence of mollifiers $(\rho_{n})$ such that for any $n\geq 1$,
\begin{equation}\label{Eq 2.7}
\rho_{n}\in C^{\infty}_{c}(\mathbb{R}^{N+1}),\; \text{Supp }\rho_{n}\subset B_{\frac{1}{n}}(0),\;\rho_{n}\geq 0 \text{ and }\int_{\mathbb{R}^{N+1}}\rho_{n}=1.
\end{equation}
Given $\mu\in\mathcal{M}(Q)$, we define $\mu_{n}$ as a convolution $\rho_{n}*\mu$ for every $(t,x)\in \mathbb{R}\times\mathbb{R}^{N}$ by 
\begin{equation}\label{Eq 2.8}
\mu_{n}(t,x)=\rho_{n}*\mu(t,x)=\int_{Q}\rho_{n}(t-s,x-y) d\mu(s,y).
\end{equation}
 For any nonnegative real number, we denote by $T_{k}(r)=\text{min}(k,\text{max}(r,-k))$ the truncation function at level $k$. For every $r\in\mathbb{R}$, let $\overline{T}_{k}(z)=\int_{0}^{z}T_{k}(s)ds$. Finally by $\langle\cdot,\cdot\rangle$ we mean the duality between suitable spaces in which functions are involved. In particular we will consider both duality between $W^{1,p}_{0}(\Omega)$ and $W^{-1,p'}(\Omega)$ and the duality between $W^{1,p}_{0}(\Omega)\cap L^{\infty}(\Omega)$ and $W^{-1,p'}(\Omega)+L^{1}(Q)$, and we denote by $\omega(h,n,\delta,\cdots)$ any quantity that vanishes as the parameters go to their limit point.
\section{Main assumptions and renormalized formulation}\label{Sect. 3}
\begin{flushleft}
Let us state our basic assumptions. Let $\Omega$ be a bounded, open subset of $\mathbb{R}^{N}$, $T$ a positive number and $Q=(0,T)\times\Omega$, we will actually consider a larger class of problems involving Leray-Lions type operators of the form $-\text{div}(a(t,x,\nabla u))$ (the same argument as above still holds for more general nonlinear operators (see \cite{BMR})), and the nonlinear parabolic problem
  \begin{equation}\label{Eq 3.1}
\left\{  
\begin{aligned}   
& b(u)_{t}-\text{div}(a(t,x,\nabla u))=\mu \quad\mbox{in }(0,T)\times\Omega,\\   
&b(u)=b(u_{0}) \quad\mbox{on }\lbrace 0\rbrace\times\Omega,\\   
&u=0 \quad\mbox{on }(0,T)\times\partial\Omega,   
\end{aligned} 
\right.  
\end{equation}
where $a:(0,T)\times\Omega\times\mathbb{R}^{N}\rightarrow\mathbb{R}^{N}$ be a Carath\'eodory function (i.e., $a(\cdot,\cdot,\zeta)$ is measurable on $Q$ for every $\zeta$ in $\mathbb{R}^{N}$, and $a(t,x,\cdot)$ is continuous on $\mathbb{R}^{N}$ for almost every $(t,x)$ in $Q$), such that the following assumptions holds:
\end{flushleft}
\begin{equation}\label{Eq 3.2}
a(t,x,\zeta)\cdot\zeta \geq \alpha|\zeta|^{p},\quad p>1, 
\end{equation}
\begin{equation}\label{Eq 3.3}
|a(t,x,\zeta)|\leq \beta[L(x,t)+|\zeta|^{p-1}], 
\end{equation}
\begin{equation}\label{Eq 3.4}
[a(t,x,\zeta)-a(t,x,\eta)]\cdot(\zeta-\eta)>0,
\end{equation}
for almost every $(t,x)$ in $Q$, for every $\zeta,\eta$ in $\mathbb{R}^{N}$, with $\zeta\neq\eta$, where $\alpha$ and $\beta$ are two positive constants, and $L$ is a nonnegative function in $L^{p'}(Q)$.\\
In all the following, we assume that $b:\mathbb{R}\rightarrow\mathbb{R}$ is a strictly increasing $C^{1}-$function which satisfies
\begin{equation}\label{Eq 3.5}
0<b_{0}\leq b'(s)\leq b_{1}\quad\forall s\in\mathbb{R}\text{ and }b(0)=0,
\end{equation}
\begin{equation}\label{Eq 3.6}
u_{0}\text{ is a measurable function in }\Omega\text{ such that }b(u_{0})\in L^{1}(\Omega),
\end{equation}
and that $\mu$ is a diffuse measure, i.e.,
\begin{equation}\label{Eq 3.7}
\mu\in\mathcal{M}_{0}(Q).
\end{equation}
Let us give the notion of renormalized solution for parabolic problem $(\ref{Eq 3.1})$ using a different formulation, we recall that the following definition is the natural extension of the one given in \cite{BPR} for diffuse measures.
\begin{definition}\label{Def 3.1}
Let $\mu\in\mathcal{M}_{0}(Q)$. A measurable function $u$ defined on $Q$ is a renormalized solution of problem $(\ref{Eq 3.1})$ if $T_{k}(b(u))\in L^{p}(0,T;W^{1,p}_{0}(\Omega))$ for every $k>0$, and if there exists a sequence $(\lambda_{k})$ in $\mathcal{M}(Q)$ such that 
\begin{equation}\label{Eq 3.8}
\underset{k\rightarrow\infty}{\text{lim}}\Vert \lambda_{k}\Vert_{\mathcal{M}(Q)}=0,
\end{equation}
and 
\begin{equation}\label{Eq 3.9}
\begin{aligned}
-\int_{Q}T_{k}(b(u))&\varphi_{t} dx dt+\int_{Q}a(t,x,\nabla u)\cdot\nabla\varphi dx dt=\\
&\int_{Q}\varphi d\mu +\int_{Q}\varphi d\lambda_{k}+\int_{\Omega}T_{k}(b(u_{0}))\varphi(0,x)dx
\end{aligned}
\end{equation}
for every $k>0$ and $\varphi\in C^{\infty}_{c}([0,T]\times\Omega)$.
\end{definition} 

\begin{remark}Note that
\begin{itemize}   
\item[(i)] Equation $(\ref{Eq 3.9})$ implies that $(T_{k}(b(u)))_{t}-\text{div}(a(t,x,\nabla u))$ is a bounded measure, and since $T_{k}(b(u))\in L^{p}(0,T;W^{1,p}_{0}(\Omega))$ and $\mu_{0}\in\mathcal{M}_{0}(Q)$ this means that 
\begin{equation}\label{Eq 3.10}
(T_{k}(b(u)))_{t}-\text{div}(a(t,x,\frac{1}{b'(u)}\nabla T_{k}(b(u))))=\mu+\lambda^{k}\text{   in   }\mathcal{M}(Q).
\end{equation}
\item[(ii)] Thanks to a result of \cite{PPP2}, the renormalized solution of problem $(\ref{Eq 3.1})$ turns out to coincide with the renormalized solution of the same problem in the sense of \cite{BPR} (see Proof of the Theorem \ref{Thm 4.2} bellow).
\item[(iii)] For every $\varphi\in W^{1,\infty}(Q)$ such that $\varphi=0$ on $(\lbrace T\rbrace\times\Omega)\cup((0,T)\times\partial\Omega)$, we can use $\varphi$ as test function in $(\ref{Eq 3.9})$ or in the approximate problem. 
\item[(iv)] A remark on the assumption $(\ref{Eq 3.5})$ is also necessary. As one could check later, due essentially to the presence of the term $g$ (dependent on $t$) in the formulation of the renormalized solution (i.e, the term with $\mu$) in Definition \ref{Def 3.1} , we are forced to assume $b'(s)\geq b_{0}>0$. We conjecture that this assumption is only technical to prove the equivalence and could be removed in order to deal with more general elliptic-parabolic problems (see \cite{AHL}, \cite{AW} and \cite{CW}). 
\end{itemize}
\end{remark}
\section{The formulation does not depend on the decomposition of $\mu$}\label{Sect. 4}
As we said before, for every measure $\mu\in\mathcal{M}_{0}(Q)$, there exist a decomposition $(f,g,\chi)$ not uniquely determined such that $f\in L^{1}(Q)$, $g\in L^{p}(0,T;V)$ and $\chi\in L^{p'}(0,T;W^{-1,p'}(\Omega))$ with 
\[ \mu=f+g+\chi\text{   in   }\mathcal{D}'(Q).\]
It is not known whether if every measure which can be decomposed in this form is diffuse. However, in \cite{PPP2} we have the following result.
\begin{lemma}
Assume that $\mu\in \mathcal{M}(Q)$ satisfies $(\ref{Eq 2.3})$, where $f\in L^{1}(Q)$, $g\in L^{p}(0,T;V)$ and $\chi\in L^{p'}(0,T;W^{-1,p'}(\Omega))$. If $g\in L^{\infty}(Q)$, then $\mu$ is diffuse.
\end{lemma}
\begin{proof}
See \cite{PPP2}, Proposition 3.1.
\end{proof}
\begin{flushleft}
Recall the notion of renormalized solution in the sense of \cite{BPR}.
\end{flushleft}
\begin{definition}\label{Def 4.1'}
Let $\mu\in\mathcal{M}_{0}(Q)$. A measurable function defined on $Q$ is a renormalized solution of problem $(\ref{Eq 3.1})$ if 
\begin{equation}\label{Eq 4.1}
 b(u)-g\in L^{\infty}(0,T;L^{1}(\Omega)),\quad T_{k}(b(u)-g)\in L^{p}(0,T;W^{1,p}_{0}(\Omega)),\quad\forall k>0,
\end{equation}
\begin{equation}\label{Eq 4.2}
\underset{h\rightarrow\infty}{\text{lim}}\int_{\lbrace h\leq |b(u)-g|\leq h+1\rbrace}|\nabla u|^{p}dx dt=0,
\end{equation}
and for every $S\in W^{2,\infty}(\mathbb{R})$ such that $S'$ has compact support, 
\begin{equation}\label{Eq 4.3}
\begin{aligned}
&-\int_{Q}S(b(u)-g)\varphi_{t}dx dt+\int_{Q}a(t,x,\nabla u)\nabla (S'(b(u)-g)\varphi)dx dt\\
&=\int_{Q}fS'(b(u)-g)\varphi dx dt+\int_{Q}G\cdot\nabla (S'(b(u)-g)\varphi)dx dt+\int_{\Omega}S(b(u_{0}))\varphi(0,x)dx,
\end{aligned}
\end{equation}
for every $\varphi\in C^{\infty}_{c}([0,T]\times\Omega)$.
\end{definition}

Finally, we conclude by proving that Definition \ref{Def 3.1} imply that $u$ is a renormlized solution in the sense of Definition \ref{Def 4.1'}, this proves that the formulations are actually equivalent.
\begin{theorem}\label{Thm 4.2}
Let $\mu$ be splitted as in $(\ref{Eq 2.3})$, namely
\[ \mu=f-\text{div }(G)+g,\quad f\in L^{1}(Q), \; G\in L^{p'}(Q)\text{ and }g\in L^{p}(0,T;V).\]
If $u$ satisfies Definition \ref{Def 3.1}, then $u$ satisfies Definition \ref{Def 4.1'}.
\end{theorem}

\begin{proof}
We split the proof in two steps.\\

\textit{Step.1} Let $v=T_{k}(b(u)-g)$, we have $v\in L^{p}(0,T;V)$. Moreover, using the decomposition of $\mu$ in $(\ref{Eq 2.3})$, and integrating by parts the term with $g$, we have
\[
\begin{aligned}
-\int_{Q}v\varphi_{t}dx dt&+\int_{Q}\frac{1}{b'(u)}a(t,x,\nabla T_{k}(b(u)))dxdt\\
&=\int_{Q}f\varphi dx dt+\int_{Q}G\cdot\nabla \varphi dx dt+\int_{Q}\varphi d\lambda_{k}+\int_{\Omega}T_{k}(b(u_{0}))\varphi(0,x)dx
\end{aligned}
\]
for every $\varphi\in C^{\infty}_{c}([0,T]\times\Omega)$. Observe that for every $\varphi\in W^{1,\infty}(Q)$ the above equality remains true. We can choose $\varphi(x,t)$ such that
\[ \varphi(x,t)=\lbrace (x,t)\frac{1}{h}\int_{t}^{t+h}\varphi(v(s,x))ds,\rbrace\]
where $\zeta\in C^{\infty}_{c}([0,T]\times\overline{\Omega})$, $\zeta\geq 0$, $\zeta\psi(0)=0$ on $(0,T)\times\partial\Omega$, and $\psi$ is lipschitz nondecreasing function. This clearly implies from (\cite{BP1}, Lemma 2.1)
\[
\begin{aligned}
&\underset{h\rightarrow 0}{\text{liminf}}\lbrace -\int_{Q}(v-T_{k}(b(u_{0})))(\zeta\frac{1}{h}\int_{t}^{t+h}\psi(v)ds)_{t}dx dt\rbrace\\
&\geq -\int_{Q}(\int_{0}^{t}\psi(r)dr)\zeta_{t}dx dt-\int_{\Omega}(\int_{0}^{T_{k}(b(u_{0}))}\psi(r)dr)\zeta(0,x)dx.
\end{aligned}
\]
Indeed, since $\psi$ is bounded, we have
\[ |\int_{Q}\psi d\lambda_{k}|\leq \Vert\zeta\Vert_{\infty}\Vert \psi\Vert_{\infty}\Vert \lambda_{k}\Vert_{\mathcal{M}(Q)},\]
and since $\psi$ is Lispchitz, we have $\psi(v)\in L^{p}(0,T;W^{1,p}_{0}(\Omega))$. Notice that $(\psi(v))_{h}$ converges to $\psi(v)$ strongly in $L^{p}(0,T;W^{1,p}_{0}(\Omega))$ and weakly* in $L^{\infty}(Q)$. So that, as $h\rightarrow 0$,
\begin{equation}\label{Eq 4.4}
\begin{aligned}
&-\int_{Q}(\int_{0}^{r}\psi(r)dr)\zeta_{r}dx dt+\int_{Q}a(t,x,\nabla T_{k}(u))\nabla (\psi(r)\zeta)dx dt\\
&\leq \int_{Q}f \psi(v)\zeta dx dt+\int_{Q}G\cdot\nabla (\psi(v)\zeta)dx dt\\
&+\int_{\Omega}(\int_{0}^{T_{k}(b(u_{0}))}\psi(r)dr)\zeta(0,x)dx +\Vert \zeta\Vert\Vert\psi\Vert_{\infty}\Vert \lambda_{k}\Vert_{\mathcal{M}(Q)},
\end{aligned}
\end{equation}
for every $\psi$ lipschitz and nondecreasing. In order to obtain the reverse inequality, we only need to take
\[ \varphi(x,t)=\lbrace (x,t)\frac{1}{h}\int_{t-h}^{t}\psi(\tilde{v}(s,x))ds\rbrace\]
where $\tilde{v}(x,t)=v(x,t)$ when $t\geq 0$ and $\tilde{v}=U_{j}$ when $t<0$, being $U_{j}\in C^{\infty}_{c}(\Omega)$ such that $U_{j}\rightarrow T_{k}(b(u_{0}))$ strongly in $L^{1}(\Omega)$. Thus, using (\cite{BP1}, Lemma 2.3), we obtain
\[
\begin{aligned}
&\underset{h\rightarrow 0}{\text{liminf}}\lbrace -\int_{Q}(v-T_{k}(b(u_{0})))(\zeta\frac{1}{h}\int_{t-h}^{t}\psi(v)ds)_{t}dx dt\rbrace\\
&\leq -\int_{Q}(\int_{0}^{r}\psi(r)dr)\zeta_{t}dx dt-\int_{\Omega}(\int_{0}^{U_{j}}\psi(r)dr)\zeta(0,x)dx\\
&-\int_{\Omega}(T_{k}(b(u_{0}))-U_{j})\zeta(0,x)dx
\end{aligned}
\]
Recalling that $\tilde{v}\in L^{p}(0,T;W^{1,p}_{0}(\Omega))\cap L^{\infty}(Q)$, when $h\rightarrow 0$, we can pass to the limit in the other terms as before, and we observe that
\[
\begin{aligned}
&-\int_{Q}(\int_{0}^{v}\psi(r)dr)\zeta_{t} dx dt+\int_{Q}a(t,x,\nabla u)\cdot\nabla (\psi(v)\zeta)dx dt\\
&\geq\int_{Q}f\psi(v)\zeta dx dt+\int_{Q}G\cdot\nabla (\psi(v)\zeta)dx dt+\int_{\Omega}(\int_{0}^{U_{0}}\psi(r)dr)\zeta(0,x)dx\\
&+\int_{\Omega}(T_{k}(b(u_{0})-U_{j})\psi(U_{j})\zeta(0,x)dx -\Vert \zeta\Vert_{\infty}\Vert \psi\Vert_{\infty}\Vert \lambda_{k}\Vert_{\mathcal{M}(Q)}
\end{aligned}
\]
Hence, from $U_{j}\rightarrow T_{k}(b(u_{0}))$, we have
\begin{equation}\label{Eq 4.5}
\begin{aligned}
&-\int_{Q}(\int_{0}^{v}\psi(r)dr)\zeta_{t}dx dt+\frac{1}{b'(u)}\int_{Q}a(t,x,\nabla T_{k}(b(u)))\cdot\nabla (\psi(r)\zeta)dx dt\\
&\geq\int_{Q}f\psi(v)\zeta dx dt+\int_{Q}G\cdot\nabla (\psi(v))dx dt+\int_{\Omega}(\int_{0}^{T_{k}(b(u_{0}))}\psi(r)dr)\zeta(0,x)dx\\
&-\Vert \zeta\Vert_{\infty}\Vert\psi\Vert_{\infty}\Vert\lambda_{k}\Vert_{\mathcal{M}(Q)}.
\end{aligned}
\end{equation}
Using equality $(\ref{Eq 4.4})$ with ($S\in W^{2,\infty}(\mathbb{R})$ and $\psi=\int_{0}^{s}(S''(t))^{+}dt$) and equality $(\ref{Eq 4.5})$ with ($\psi=\int_{0}^{s}(S''(t))^{-}dt$), we easily deduce by substracting the two inequalities (observe that $S'(s)=\int_{0}^{s}(S''(t)^{+}-S''(t)^{-})dt$) that
\begin{equation}\label{Eq 4.6}
\begin{aligned}
&-\int_{Q}S(v)\zeta_{t}dx dt +\int_{Q}a (t,x,\nabla u)\cdot\nabla (S'(v)\zeta)dx dt\\
&\leq \int_{Q}fS'(v)\zeta dx dt+\int_{Q}G\nabla (S'(v)\zeta)dx dt\\
&+\int_{\Omega}S(T_{k}(b(u_{0})))\zeta(0,x)dx +2\Vert \zeta\Vert_{\infty}\Vert S'\Vert_{\infty}\Vert\lambda_{k}\Vert_{\mathcal{M}(Q)},
\end{aligned}
\end{equation}
for every $S\in W^{2,\infty}(\mathbb{R})$ and for every nonnegative $\zeta$.\\

\textit{Step.2} Let us use $S'(\Theta_{h}(s))$ in $(\ref{Eq 4.6})$ such that $\Theta_{h}=T_{1}(s-T_{h}(s))$ and $\zeta=\zeta(t)$. Then we easily obtain by setting $R_{h}(s)=\int_{0}^{s}\Theta_{h}(\zeta)d\zeta$,
\[
\begin{aligned}
&-\int_{Q}R_{h} (T_{k}(b(u))-g)\zeta_{t}dx dt+\int_{\lbrace h<|b(u)-g|<h+k\rbrace}a(t,x,\nabla u)\cdot\nabla (T_{k}(b(u))-g)\zeta dx dt\\
&\leq \int_{Q}f \Theta_{h}(T_{k}(b(u))-g)\zeta dx dt+\int_{\lbrace h<|b(u)-g|< h+k\rbrace}G\cdot\nabla (T_{k}(b(u)-g))dx dt\\
&+\int_{\Omega}R_{h}(T_{k}(b(u_{0})))\zeta(0,x)dx +2\Vert \zeta\Vert_{\infty}\Vert\lambda_{k}\Vert_{\mathcal{M}(Q)}.
\end{aligned}
\]
Moreover, we can use young's inequality, assumption $(\ref{Eq 3.2})$ and $(\ref{Eq 3.3})$ to get
\[
\begin{aligned}
&-\int_{Q}R_{h}(T_{k}(b(u)-g))\zeta_{t}dx dt+\int_{\lbrace h<|b(u)-g|< h+1\rbrace}b'(u)|\nabla T_{k}(b(u))|^{p}\zeta dx dt\\
&\leq\int_{Q}f\Theta_{h}(T_{k}(b(u))-g)\zeta dx dt +C\int_{\lbrace h<|b(u)-g|<h+1\rbrace}(|G|^{p'}+|g|^{p}+|L|^{p'})\zeta dx dt\\
&+\int_{\Omega}R_{h}(T_{k}(b(u_{0})))\zeta(0,x)dx +2\Vert \zeta\Vert_{\infty}\Vert\lambda_{k}\Vert_{\mathcal{M}(Q)},
\end{aligned}
\]
Now, letting $k\rightarrow\infty$, thanks to $(\ref{Eq 3.8})$  and Fatou's Lemma, we deduce 
\[
\begin{aligned}
&-\int_{Q}R(b(u)-g)\zeta_{t}dx dt+\alpha\int_{\lbrace h<|b(u)-g|<h+1\rbrace}b'(u)|\nabla u|^{p}dx dt \\
&\leq \int_{Q}f \Theta_{h}(u-g)\zeta dx dt +C\int_{\lbrace h<|b(u)-g|\leq h+1\rbrace}(|G|^{p'}+|g|^{p}+|L|^{p'})\zeta dx dt\\
&+\int_{\Omega}R_{h}(b(u_{0}))\zeta(0,x)dx
\end{aligned}
\]
Consider $\zeta=1-\frac{1}{\epsilon}T_{\epsilon}(t-\tau)^{+}$, for $\tau\in (0,T)$, and letting $\epsilon\rightarrow 0$, we claim that the estimate of $b(u)-g$ in $L^{\infty}(0,T;L^{1}(\Omega))$ is valid. By repeating the argument for the nonincreasing $\zeta_{\epsilon}\in C^{\infty}_{c}([0,T])$, we are allowed to pass to the limit  $\zeta_{\epsilon}\rightarrow 1$ to prove that
\[
\begin{aligned}
&\alpha b_{0}\int_{\lbrace h<|b(u)-g|<h+1\rbrace}|\nabla u|^{p}dx dt\\
&\leq\int_{\lbrace |b(u)-g|>h\rbrace}|f|dx dt+C\int_{\lbrace h<|b(u)-g|<h+1\rbrace}(|G|^{p'}+|g|^{p}+|L|^{p'})\zeta dx dt +\int_{\lbrace|b(u_{0})|>h\rbrace}b(u_{0})dx,
\end{aligned}
\]
which implies $(\ref{Eq 4.2})$. Finally, by using $S\in W^{2,\infty}(\mathbb{R})$ such that $S'$ has compact support, $\zeta\in C^{\infty}_{c}([0,T]\times\Omega)$ and the regularity  $(\ref{Eq 4.1})$, we can easily deduce $(\ref{Eq 4.3})$ by passing to the limit in $(\ref{Eq 4.6})$  and using $(\ref{Eq 3.8})$.
\end{proof}
\section{Existence of Solutions}\label{Sect. 5}
Now we are ready to prove the main results. Some of the reasoning is based on the ideas developed in \cite{BPR} (see also \cite{DPP}, \cite{PPP2} and \cite{Po1}). First we have to prove the existence of renormalized solution for problem $(\ref{Eq 3.1})$.\\
\begin{theorem}\label{Thm 5.1}
Under assumptions $(\ref{Eq 3.1})-(\ref{Eq 3.7})$, there exists at least a renormalized solution $u$ of problem $(\ref{Eq 3.1})$.
\end{theorem}
\begin{proof}
We first introduce the approximate problems. For $n\geq 1$ fixed, we define
\begin{equation}\label{Eq 5.1}
b_{n}(s)=b(T_{\frac{1}{n}}(s))+ns\text{   a.e. in   }\Omega,\; \forall s\in\mathbb{R},
\end{equation}
\begin{equation}\label{Eq 5.2}
u_{0}^{n}\in C^{\infty}_{0}(\Omega):\quad b_{n}(u_{0}^{n})\rightarrow b(u_{0})\text{ in }L^{1}(\Omega)\text{ as n tends to }+\infty. 
\end{equation}
We consider a sequence of mollifiers $(\rho_{n})$, and we define the convolution $\rho_{n}*\mu$ for every $(t,x)\in Q$ by 
\begin{equation}\label{Eq 5.3}
\mu^{n}(t,x)=\rho_{n}*\mu(t,x)=\int_{Q}\rho_{n}(t-s,x-y) d\mu(s,y).
\end{equation}
Then we consider the approximate problem of $(\ref{Eq 3.1})$
\begin{equation}\label{Eq 5.4}
\left\{  
\begin{aligned}   
&(b_{n}(u_{n}))_{t}-\text{div}(a(t,x,\nabla u_{n}))=\mu_{n} \quad\mbox{in }(0,T)\times\Omega\\   
&b_{n}(u_{n})=b_{n}(u_{0}^{n}) \quad\mbox{on }\lbrace 0\rbrace\times\Omega\\   
&u_{n}=0 \quad\mbox{on }(0,T)\times\partial\Omega,   
\end{aligned} 
\right.  
\end{equation}
By classical results (see \cite{L}), we can find a nonnegative weak solution $u_{n}\in L^{p}(0,T;W^{1,p}_{0}(\Omega))$ for problem $(\ref{Eq 5.4})$. Our aim is to prove that a subsequence of these approximate solutions $(u_{n})$ converges increasingly to a measurable function $u$, which is a renormalized solution of problem $(\ref{Eq 3.1})$. We will divide the proof into several steps. We present a self-contained proof for the sake of clarity and readability.\\

\textbf{Step.1} Basic estimates.\\
 Choosing $T_{k}(b_{n}(u_{n})-g_{n})$ as a test function in $(\ref{Eq 5.4})$, we have 
\begin{equation}\label{Eq 5.5}
\begin{aligned}
&\int_{\Omega}\overline{T}_{k}(b_{n}(u_{n})-g_{n})dx +\int_{0}^{t}\int_{\Omega}a(x,s,\nabla u_{n})\cdot\nabla T_{k}(b_{n}(u_{n})-g_{n})dx ds=\\
&\int_{0}^{t}\int_{\Omega}f_{n}T_{k}(b_{n}(u_{n})-g_{n})dx dt+\int_{0}^{t}\int_{\Omega}G_{n}\cdot\nabla T_{k}(b_{n}(u_{n})-g_{n})dx ds+\int_{\Omega}\overline{T}_{k}(b_{n}(u_{0}^{n}))dx,
\end{aligned}
\end{equation}
for almost every $t$ in $(0,T)$, and where $\overline{T}_{k}(r)=\int_{0}^{r}T_{k}(s)ds$. It follows from the definition of $\overline{T}_{k}$, assumptions $(\ref{Eq 3.2})-(\ref{Eq 3.3})$ and $(\ref{Eq 3.6})$ that 
\begin{equation}\label{Eq 5.6}
\begin{aligned}
&\int_{\Omega}\overline{T}_{k}(b_{n}(u_{n})-g_{n})dx +\alpha\int_{\lbrace|b_{n}(u_{n})-g_{n}|\leq k\rbrace}b'_{n}(u_{n})|\nabla u_{n}|^{p}dx ds\\
&\leq k\Vert \mu_{n}\Vert_{L^{1}(Q)}+\beta\int_{\lbrace|b_{n}(u_{n})-g_{n}|\leq k\rbrace}L(x,s)|\nabla g_{n}|dx ds\\
&+\beta\int_{\lbrace|b_{n}(u_{n})-g_{n}|\leq k\rbrace}|\nabla u_{n}|^{p-1}|\cdot\nabla g_{n}|dx ds+k\Vert b_{n}(u_{0}^{n})\Vert_{L^{1}(\Omega)}
\end{aligned}
\end{equation}
Then, from $(\ref{Eq 3.5})$ and young's inequality 
\begin{equation}\label{Eq 5.7}
\begin{aligned}
&\int_{\Omega}\overline{T}_{k}(b_{n}(u_{n})-g_{n})dx +\frac{\alpha}{2}\int_{\lbrace|b_{n}(u_{n})-g_{n}|\leq k\rbrace}b'_{n}(u_{n})|\nabla u_{n}|^{p}dx dt\\
&\leq k\Vert \mu_{n}\Vert_{L^{1}(Q)}+\beta\Vert L\Vert_{L^{p'}(Q)}\Vert\nabla g_{n}\Vert_{L^{p}(Q)}\\
&+C\Vert \nabla g_{n}\Vert^{p'}_{L^{p'}(Q)}+ k\Vert b_{n}(u_{0}^{n})\Vert_{L^{1}(\Omega)}
\end{aligned}
\end{equation}
where $C$ is a positive constant. We will use the properties of $\overline{T}_{k}$ ($\overline{T}_{k}\geq 0$, $\overline{T}_{k}(s)\geq |s|-1$, $\forall s\in\mathbb{R}$), $b_{n}$, $f_{n}$, $G_{n}$, $g_{n}$, the boundedness of $\mu_{n}$ in $L^{1}(Q)$  and $b_{n}(u_{0}^{n})$ in $L^{1}(\Omega)$ to have 
\begin{equation}\label{Eq 5.8}
b_{n}(u_{n})-g_{n}\text{   is bounded in   }L^{\infty}(0,T;L^{1}(\Omega))
\end{equation}
Using H\"older inequality and $(\ref{Eq 3.5})$, we deduce that $(\ref{Eq 5.7})$ implies 
\begin{equation}\label{Eq 5.9}
T_{k}(b_{n}(u_{n})-g_{n})\text{   is bounded in   }L^{p}(0,T;W^{1,p}_{0}(\Omega)), 
\end{equation}
Independently of $n$ for any $k\geq 0$.\\

Let us observe from (\cite{BM} and \cite{BMR}) that for any $S\in W^{2,\infty}(\mathbb{R})$ such that $S'$ has a compact support $(\text{supp}(S')\subset[-k,k])$
\begin{equation}\label{Eq 5.10}
 S(b_{n}(u_{n})-g_{n})\text{   is bounded in   }L^{p}(0,T;W^{1,p}_{0}(\Omega)),
\end{equation}
and 
\begin{equation}\label{Eq 5.11}
 (S(b_{n}(u_{n})-g_{n}))_{t}\text{   is bounded in   }L^{1}(Q)+L^{p'}(0,T;W^{-1,p'}(\Omega)).
\end{equation}
independently of $n$. In fact, thanks to $(\ref{Eq 5.9})$ and Stampacchia's theorem, we easily deduce $(\ref{Eq 5.10})$. To show that $(\ref{Eq 5.11})$ hold true, we multiply $(\ref{Eq 5.4})$ by $S'(b_{n}(u_{n})-g_{n})$ to obtain
\begin{equation}\label{Eq 5.12}
\begin{aligned}
(S(b_{n}(u_{n})-g_{n}))_{t}&=\text{div}(S(b_{n}(u_{n})-g_{n})a(t,x,\nabla u_{n}))\\
&-a(t,x,\nabla u_{n})\cdot\nabla S'(b_{n}(u_{n})-g_{n})+f_{n}S'(b_{n}(u_{n})-g_{n})\\
&-\text{div}(G_{n}S'(b_{n}(u_{n})-g_{n}))+G_{n}\cdot\nabla S(b_{n}(u_{n})-g_{n})\text{   in   }\mathcal{D}'(Q),
\end{aligned}
\end{equation}
as a consequencen each term in the right hand side of $(\ref{Eq 5.12})$ is bounded either in $L^{p'}(0,T;W^{-1,p'}(\Omega))$ or in $L^{1}(Q)$, we obtain $(\ref{Eq 5.11})$.\\

Moreover, arguing again as in \cite{BPR} (see also \cite{BM}, \cite{BMR} and \cite{BR}), there exists a measurable function $u$ such that $T_{k}(u)\in L^{p}(0,T;W^{1,p}_{0}(\Omega))$, $u$ belongs to $L^{\infty}(0,T;L^{1}(\Omega))$, and up to a subsequence, for any $k>0$ we have  
\begin{equation}\label{Eq 5.13}
\left\{
\begin{aligned}
&u_{n}\rightarrow u \text{ a.e. in }Q,\\
& T_{k}(u_{n})\rightharpoonup T_{k}(u)\text{ weakly in }L^{p}(0,T;W^{1,p}_{0}(\Omega)),\\
& b_{n}(u_{n})-g_{n}\rightarrow b(u)-g\text{ a.e. in }Q,\\
&T_{k}(b_{n}(u_{n})-g_{n})\rightharpoonup T_{k}(b(u)-g)\text{ weakly in }L^{p}(0,T;W^{1,p}_{0}(\Omega)), 
\end{aligned}
\right.
\end{equation}
as $n$ tends to $+\infty$.\\

\textit{Step.2} Estimates in $L^{1}(Q)$ on the energy term.\\
Let $\rho_{n}$ a sequence of mollifiers as in $(\ref{Eq 2.7})$ and $\mu$ a nonnegative measure such that 
$\mu_{n}(t,x)=\rho_{n}*\mu(t,x)$. Observe that, based on Lemma \ref{Lem 2.4} that $\mu_{n}$ is an equidiffuse sequence of measures. Moreover, there exists a sequence $\mu_{n}\in C^{\infty}(Q)$ such that 
\[\Vert \mu\Vert_{L^{1}(Q)}\leq \Vert\mu\Vert_{\mathcal{M}(Q)},\]
and 
\[\mu_{n}\rightarrow \mu \text{ tightly in }\mathcal{M}(Q). \]
Let us fix $\eta>0$ and define $S_{k,\eta}(s):\mathbb{R}\rightarrow\mathbb{R}$  and $h_{k,\eta}(s):\mathbb{R}\rightarrow\mathbb{R}$ by 

\begin{equation}\label{Eq 5.14}
 S_{k,\eta}(s)=  \begin{cases}
 1 & \text{ if }|s|\leq k \\
0 & \text{ if }|s|> k+\eta \\
\text{affine}& \text{ otherwise}
    \end{cases}
\quad \text{ and }\quad h_{k,\eta}(s) =  1-S_{k,\eta}(u_{n}), 
\end{equation}
 let us denote by $T_{k,\eta}:\mathbb{R}\rightarrow\mathbb{R}$ the primitive function of $S_{k,\eta}$, that is 
\[ T_{k,\eta}(s)=\int_{0}^{s}S_{k,\eta}(\sigma)d\sigma\]
Notice that $T_{k,\eta}(s)$ converges pointwise to $T_{k}(s)$ as $\eta$ goes to zero and using the admissible test function $h_{k,\eta}(b(u_{n}))$ in $(\ref{Eq 5.4})$ leads to 
\begin{equation}\label{Eq 5.15}
\begin{aligned}
&\int_{\Omega}\overline{h}_{k,\eta}(b(u_{n})(T))dx+\frac{1}{\eta}\int_{\lbrace k<u_{n}<k+\eta\rbrace}a(t,x,\nabla u_{n})\nabla h_{k,\eta}(b(u_{n}))\\
&=\int_{Q}h_{k,\eta}(b(u_{n}))\mu_{n}dx dt+\int_{\Omega}\overline{h}_{k,\eta}b(u_{0}^{n})dx, 
\end{aligned}
\end{equation}
where $\overline{h}_{k,\eta}(r)=\int_{0}^{r}h_{k,\eta}(s)ds\geq 0$. Hence, using $(\ref{Eq 5.2})$, $(\ref{Eq 5.3})$ and dropping a nonnegative term,
\begin{equation}\label{Eq 5.16}
\begin{aligned}
&\frac{1}{\eta}\int_{\lbrace k<b(u_{n})<k+\eta\rbrace}b'(u_{n})a(t,x,\nabla u_{n})\cdot\nabla u_{n}dx ds\\
&\leq \int_{\lbrace|b(u_{n})|>k\rbrace}|\mu_{n}|dx dt +\int_{\lbrace b(u_{0}^{n})>k\rbrace}|b(u_{0}^{n})|dx \leq C.
\end{aligned} 
\end{equation}
Thus, there exists a bounded Radon measures $\lambda_{k}^{n}$ such that, as $\eta$ tends to zero 
\begin{equation}\label{Eq 5.17}
\lambda_{k}^{n,\eta}=\frac{1}{\eta}a(t,x,\nabla u_{n})\cdot\nabla u_{n}\chi_{\lbrace k\leq b(u_{n})\leq k+\eta\rbrace}\rightharpoonup \lambda_{k}^{n} *\text{weakly in }\mathcal{M}(Q). 
\end{equation}
\\
\textit{Step.3} Equation for the truncations.\\
We are able to prove that $(\ref{Eq 3.9})$ holds true. To see that, we multiply $(\ref{Eq 5.4})$ by $S_{k,\eta}(b(u_{n}))\xi$ where $\xi\in C^{\infty}_{c}([0,T]\times\Omega)$ to obatin
\begin{equation}\label{Eq 5.18}
\begin{aligned}
& T_{k,\eta}(b(u_{n}))_{t}-\text{div}(S_{k,\eta}(b(u_{n}))a(t,x,\frac{1}{b'(u_{n})}\nabla T_{k,\eta}(b(u_{n}))))\\
&=\mu_{n}+(S_{k,\eta}(b(u_{n}))-1)\mu_{n}+\frac{1}{n}a(t,x,\nabla u_{n})\cdot\nabla u_{n}\chi_{\lbrace k<|b(u_{n})|<k+\eta\rbrace}\text{ in }\mathcal{D}'(Q).
\end{aligned}
\end{equation}
Passing to the limit in $(\ref{Eq 5.18})$ as $\eta$ tends to zero, and using the fact that $|S_{k,\eta}|\leq 1$ and $(\ref{Eq 5.17})$, we deduce 
\begin{equation}\label{Eq 5.19}
T_{k}(b(u_{n}))_{t}-\text{div}(a(t,x,\frac{1}{b'(u)}\nabla T_{k}(b(u_{n}))))=\mu_{n}-\mu_{n}\chi_{\lbrace |b(u_{n})|\leq k\rbrace}+\lambda_{k}^{n}\text{ in }\mathcal{D'}(Q).
\end{equation}
Now, using properties of the convolution $\rho_{n}*\mu$ and in  view of $(\ref{Eq 5.16})-(\ref{Eq 5.17})$, we deduce that $\Lambda_{k}^{n}=-\mu_{n}\chi_{\lbrace |b(u_{n})|<k\rbrace}+\lambda_{k}^{n}$ is bounded in $L^{1}(Q)$. Then there exists a bounded measures $\Lambda_{k}$ such that $(-\mu_{n}\chi_{\lbrace |b(u_{n})|<k\rbrace}+\lambda_{k}^{n})_{n}$ converges to $\Lambda_{k}$ *weakly in $\mathcal{M}(Q)$. Therefore, using results $(\ref{Eq 5.13})$ of \textit{Step.1} and $(\ref{Eq 5.19})$ we deduce that $u$ satisfies
\begin{equation}\label{Eq 5.20}
T_{k}(b(u))_{t}-\text{div}(a(t,x,\nabla u))=\mu+\Lambda_{k}\text{ in }\mathcal{D}'(Q).
\end{equation}
\\
\textit{Step.4} $u$ is a renormalized solution.\\
  In this step, $\Lambda_{k}$ is shown to satisfy $(\ref{Eq 3.8})$. From $(\ref{Eq 5.16})$ and $(\ref{Eq 5.17})$ we deduce
 \begin{equation}\label{Eq 5.21}
 \begin{aligned}
 &\Vert \Lambda_{k}^{n}\Vert_{L^{1}(Q)}=\Vert -\mu_{n}\chi_{\lbrace |b(u_{n})|>k\rbrace}+\lambda_{k}^{n}\Vert_{L^{1}(Q)}\\
 &\leq 2\int_{\lbrace |b(u_{n})|>k\rbrace}|\mu_{n}|dx dt+\int_{\lbrace |b(u_{0}^{n})|>k\rbrace}|b(u_{0}^{n})|dx.
 \end{aligned}
 \end{equation}
Since
\[ \Vert \lambda_{k}\Vert_{\mathcal{M}(Q)}\leq\underset{n\rightarrow +\infty}{\text{liminf }}\Vert \mu_{n}\chi_{\lbrace |b(u_{n})|>k\rbrace}+\lambda_{k}^{n}\Vert_{\mathcal{M}(Q)},\]
the sequence $(\mu_{n})$ is equidiffuse, and the function $b(u_{0}^{n})$ converges to $b(u_{0})$ strongly in $L^{1}(\Omega)$, we deduce from Proposition \ref{Prop 2.2} and $(\ref{Eq 5.21})$ that $\Vert \Lambda_{k}\Vert_{\mathcal{M}(Q)}$ tends to zero as $k$ tends to infinity, then we obtain $(\ref{Eq 3.8})$, and hence, $u$ is a renormalized solution.
\end{proof}

\section{Uniqueness of renormalized solution}\label{Sect. 6}
This section is devoted to establish the uniqueness of the renormalized solution.\\
As we already said, due to the presence of both the general monotone operator associated to $a$ and the nonlinearity of the term $b$, a standard approach (see for instance \cite{DPP}) does not apply here. To overcome this difficulty, we are going to exploit the idea of \cite{PPP2} for which the uniqueness result comes from the following comparaison principle.
\begin{theorem}\label{Thm 6.1}
Let $u_{1}, u_{2}$ be two renormalized solutions of problem $(\ref{Eq 3.1})$ with data $(b(u_{0}^{1}),\mu_{1})$ and $(b(u_{0}^{2}),\mu_{2})$ respectively. Then, we have
\begin{equation}\label{Eq 6.1}
 \int_{\Omega}(b(u_{1})-b(u_{2}))^{+}(t)dx \leq \Vert b(u_{0}^{1})-b(u_{0}^{2})\Vert_{L^{1}(\Omega)}+\Vert (\mu_{1}-\mu_{2})^{+}\Vert_{\mathcal{M}(Q)}
 \end{equation}
for almost every $t\in [0,T]$. In particular, if $b(u_{0}^{1})\leq b(u_{0}^{2})$ and $\mu_{1}\leq \mu_{2}$ (in the case of measures), we have $u_{1}\leq u_{2}$ a.e. in $Q$. As a consequence, there exists at lest one renormalized solution of problem $(\ref{Eq 3.1})$.
\end{theorem}
\begin{proof}
Let $\lambda_{k_{1}},\lambda_{k_{2}}$ be the measures given by Definition \ref{Def 3.1} corresponding to $b(u_{1}), b(u_{2})$, we can extend the class of test functions 
\[
\begin{aligned}
&-\int_{Q}(T_{k}(b(u_{1}))-T_{k}(b(u_{2}))v_{t}dx dt +\int_{Q}(a(t,x,\nabla u_{1})-a(t,x,\nabla u_{2}))\cdot\nabla v dx dt\\
&=\int_{Q}vd(\mu_{1}-\mu_{2})+\int_{Q}v d\lambda_{k,1}-\int_{Q}v d\lambda_{k,2}+\int_{\Omega}(T_{k}(b(u_{0}^{1}))-T_{k}(b(u_{0}^{2})))v(0,x)dx,
\end{aligned}
\]
for every $v\in W\cap L^{\infty}(Q)$, such that $v(T)=0$. Consider the function
\[\omega_{h}(t,x)=\frac{1}{h}\int_{t}^{t+h}\frac{1}{\epsilon}T_{\epsilon}(T_{k}(b(u_{1}))-T_{k}(b(u_{2})))^{+}(s,x)ds. \]
Given $\zeta\in C^{\infty}_{c}([0,T)),\zeta\geq 0$, take $v=\omega_{h}\zeta$ as test function. Observe that both $\omega_{h}$ and $(\omega_{h})_{t}$ belong to $L^{p}(0,T;V)\cap L^{\infty}(Q)$ for $h>0$ sufficiently small, hence $\omega_{h}\in W\cap L^{\infty}(Q)$. Moreover, we have
\[ \omega_{h}\rightarrow \frac{1}{\epsilon}T_{\epsilon}(T_{k}(b(u_{1}))-T_{k}(b(u_{2})))^{+}\text{ strongly in }L^{p}(0,T;W^{1,p}_{0}(\Omega)).\]
Using that $0\leq \omega_{h}\leq 1$ almost everywhere , hence $0\leq \omega_{h}\leq 1$ $\text{cap}$-quasi-everywhere (see \cite{DPP}), we have
\begin{equation}\label{Eq 6.2}
\begin{aligned}
&-\int_{Q}[(T_{k}(b(u_{1}))-T_{k}(b(u_{2})) -(T_{k}(b(u_{0}^{1}))-T_{k}(b(u_{0}^{2}))](\omega_{h}\zeta)_{t}dx dt\\
&+\int_{Q}(a(t,x,\nabla u_{1})-a(t,x,\nabla u_{2}))\cdot\nabla \omega_{h}\zeta dx dt\\
&\leq \Vert \zeta\Vert_{\infty}(\Vert(\mu_{1}-\mu_{2})^{+}\Vert_{\mathcal{M}(Q)}+\Vert \lambda_{k,1}\Vert_{\mathcal{M}(Q)}+\Vert \lambda_{k,2}\Vert_{\mathcal{M}(Q)}).
\end{aligned}
\end{equation}
Using the monotonicity of $T_{\epsilon}(s)$, we have (see \cite{BP1}, Lemma 2.1)
\[
\begin{aligned}
&\underset{h\rightarrow 0}{\text{liminf}}\lbrace -\int_{Q}[(T_{k}(b(u_{1}))-T_{k}(b(u_{2}))-(T_{k}(b(u_{0}^{1}))-T_{k}(b(u_{0}^{2})))](\omega_{h}\zeta_{t})dx dt\rbrace\\
&\geq -\int_{Q}\tilde{\Theta}_{\epsilon}(T_{k}(b(u_{1}))\zeta_{t}dx dt -\int_{\Omega}\tilde{\Theta}_{\epsilon}(T_{k}(b(u_{0}^{1}))-T_{k}(b(u_{0}^{2}))\zeta(0)dx
\end{aligned}
\]
where $\tilde{\Theta}_{\epsilon}(s)=\int_{0}^{s}\frac{1}{\epsilon}T_{\epsilon}(r)^{+}dr$. Therefore, letting $h\rightarrow 0$ in $(\ref{Eq 6.2})$, we obtain 
\[
\begin{aligned}
&-\int_{Q}\tilde{\Theta}_{\epsilon}(T_{k}(b(u_{1}))-T_{k}(b(u_{2}))\zeta_{t}dx dt\\
&+\frac{1}{\epsilon}\int_{Q}(a(t,x,\nabla u_{1})-a(t,x,\nabla u_{2}))\cdot\nabla T_{\epsilon}(T_{k}(b(u_{1}))-T_{k}(b(u_{2}))\zeta dx dt\\
&\leq\int_{\Omega}\tilde{\Theta}_{\epsilon}(T_{k}(b(u_{0}^{1}))-T_{k}(b(u_{0}^{2}))\zeta(0)dx+\Vert \zeta\Vert_{\infty}(\Vert(\mu_{1}-\mu_{2})^{+}\Vert_{\mathcal{M}(Q)}+\Vert \lambda_{k,1}\Vert_{\mathcal{M}(Q)}+\Vert \lambda_{k,2}\Vert_{\mathcal{M}(Q)}).
\end{aligned}
\]
Using $(\ref{Eq 3.4})$ and letting $\epsilon\rightarrow 0$, we deduce
\[
\begin{aligned}
-\int_{Q}(T_{k}(b(u_{1}))-T_{k}(b(u_{2}))^{+}\zeta_{t}dx dt &\leq \int_{\Omega}(T_{k}(b(u_{0}^{1}))-T_{k}(b(u_{0}^{2}))^{+}\zeta(0)dx\\
&+\Vert\zeta\Vert_{\infty}(\Vert(\mu_{1}-\mu_{2})^{+}\Vert_{\mathcal{M}(Q)}+\Vert\lambda_{k,1}\Vert_{\mathcal{M}(Q)}+\Vert\lambda_{k,2}\Vert_{\mathcal{M}(Q)}).
\end{aligned}
\]
and letting $k\rightarrow\infty$, we obtain, thanks to $(\ref{Eq 3.8})$,
\[-\int_{Q}(b(u_{1})-b(u_{2}))^{+}\zeta_{t}dx dt\leq\Vert\zeta\Vert_{\infty}(\Vert(b(u_{0}^{1})-b(u_{0}^{2})^{+}\Vert_{L^{1}(\Omega)}+\Vert(\mu_{1}-\mu_{2})^{+}\Vert_{\mathcal{M}(Q)}) \]
for every nonnegative $\zeta\in C^{\infty}_{c}([0,T)$. Of course, the same inequality holds for any $\zeta\in W^{1,\infty}(0,T)$ with compact support in $[0,T)$. Take then $\zeta(t)=1-\frac{1}{\epsilon}T_{\epsilon}(t-\tau)^{+}$, where $\tau\in (0,T)$; since $b(u_{1}),b(u_{2})\in L^{\infty}(0,T;L^{1}(\Omega))$, by letting $\epsilon\rightarrow 0$, we have
\[-\int_{Q}(b(u_{1})-b(u_{2}))^{+}\zeta_{t}dx dt=\frac{1}{\epsilon}\int_{\tau}^{\tau+\epsilon}\int_{\Omega}(b(u_{1})-b(u_{2}))^{+}dx dt\rightarrow\int_{\Omega}(b(u_{1})-b(u_{2}))^{+}(\tau)dx \]
for almost every $\tau\in (0,T)$. Using in the right-hand side that $\Vert\zeta\Vert_{\infty}\leq 1$, we get $(\ref{Eq 6.1})$.
\end{proof}
\section{Appendix}\label{Sect. 7}

Here we proof the extension of Proposition \ref{Prop 2.2}.
\begin{proof} We still use the notations introduced in Section 2, in particular, we consider the condition $p>\frac{2N+1}{N+1}$, for simplicity we assume in addition that $\mu\geq 0$ and $b(u_{0})\geq 0$, hence, we have $u\geq 0$ (th case $\mu\leq 0$ can be obtained similarly). Actually, the proof will be split into three parts, we begin with the first one to obtain the basic estimates. \\

\textit{Step.1} Estimates of $T_{k}(b(u))$ in the space $L^{\infty}(0,T;L^{2}(\Omega))\cap L^{p}(0,T;W^{1,p}_{0}(\Omega))$.\\
For every $\tau\in\mathbb{R}$, let 
\[\overline{T_{k}}(r)=\int_{0}^{r}T_{k}(s)ds. \]
We recall that if $u\in W$, then $u$ is a weak solution of $(\ref{Eq 1.1})$ if 
\begin{equation}\label{Eq 7.1}
\int_{0}^{t}\langle b(u)_{t},v\rangle dt+\int_{Q}|\nabla u|^{p-2}\nabla u\cdot\nabla v dx dt=\int_{0}^{t}\langle \mu,v\rangle dt,\quad\forall v\in W,
\end{equation}
where $\langle\cdot,\cdot\rangle$ denotes the duality between $V$ and $V'$.\\
Note that, if $\mu\in\mathcal{M}(Q)\cap L^{p'}(0,T;W^{-1,p'}(\Omega))$, then $(\ref{Eq 7.1})$ holds for every $v\in L^{p}(0,T;V)$, and we have
\begin{equation}\label{Eq 7.2}
\int_{s}^{t}\langle b(u)_{t},\psi'(u)\rangle dt =\int_{\Omega}\psi(b(u)(t)) dx -\int_{\Omega}\psi(b(u)(s))dx,
\end{equation}
for every $s,t\in [0,T]$ and every function $\psi:\mathbb{R}\rightarrow\mathbb{R}$ such that $\psi'$ is Lipschitz continuous and $\psi'(0)=0$. Now we choose as test function $T_{k}(b(u))$ in $(\ref{Eq 7.1})$ and using $(\ref{Eq 7.2})$ with $\psi=\overline{T_{k}}$, $s=0$ and $t=r$, we have
\[ \int_{\Omega}\overline{T}_{k}(b(u))(r)dx +\int_{0}^{r}\int_{\Omega}a(t,x,\nabla u)\cdot\nabla T_{k}(b(u))dx dt\leq k\Vert\mu\Vert_{\mathcal{M}(Q)}+\int_{\Omega}\overline{T}_{k}(b(u_{0}))dx.\]
Let $E_{k}=\lbrace(x,t):|b(u)|\leq k\rbrace$, and observing $\frac{T_{k}(s)^{2}}{2}\leq\overline{T}_{k}(s)\leq k|s|$, $\forall s\in\mathbb{R}$, we have
\begin{equation}\label{Eq 7.3}
\int_{\Omega}\frac{|T_{k}(b(u))(r)|^{2}}{2}dx +\int_{0}^{r}\int_{\Omega}\chi_{E_{k}}b'(u)a(t,x,\nabla u)\cdot\nabla u \;dx dt\leq k(\Vert \mu\Vert_{\mathcal{M}(Q)}+\Vert b(u_{0})\Vert_{L^{1}(\Omega)}),
\end{equation}
for any $r\in [0,T]$. In particulier, we deduce
\begin{equation}\label{Eq 7.4}
\Vert T_{k}(b(u))\Vert^{2}_{L^{\infty}(0,T;L^{2}(\Omega))}\leq 2kM,
\end{equation}
and from assumption (3.2), we have
\[\alpha\int_{E_{k}}b'(u)|\nabla u|^{p}dx dt\leq \int_{0}^{r}\int_{\Omega}\chi_{E_{k}}b'(u)a(t,x,\nabla u)\nabla u\leq kM. \]
Note that
\[
\begin{aligned}
\int_{E_{k}}b'(u)|\nabla u|^{p}dx dt&=\int_{E_{k}}b'(u)|b'^{-1}\nabla b(u)|^{p}dx\\
&=\int_{E_{k}}\frac{1}{(b')^{p-1}}|\nabla b(u)|^{p}dx\\
&\geq \int_{0}^{r}\int_{\Omega}\frac{1}{(b_{1})^{p-1}}|\nabla T_{k}b(u)|^{p}dx.
\end{aligned}
\]
Then,
\begin{equation}\label{Eq 7.5}
\Vert T_{k}(b(u))\Vert^{p}_{L^{p}(0,T;W^{1,p}_{0}()\Omega)}\leq CkM, 
\end{equation}
where 
\begin{equation}\label{Eq 7.6}
 C=\frac{b_{1}^{p-1}}{\alpha}\quad\text{and}\quad M=\Vert \mu\Vert_{\mathcal{M}(Q)}+\Vert b(u_{0})\Vert_{L^{1}(\Omega)}.
 \end{equation}\\
 
\textit{Step.2} Estimates in $W$.\\
Note that in virtue of \cite{P} (see also \cite{L}), any function $z\in L^{\infty}(0,T;L^{2}(\Omega))\cap L^{p}(0,T;W^{1,p}_{0}(\Omega))$ is a solution of the backward problem
\begin{equation}\label{Eq 7.7}
\left\{
\begin{aligned}
&-z_{t}-\Delta_{p}z=-2\Delta_{p}T_{k}(b(u))\quad\text{ in }Q,\\
&z=T_{k}(b(u))\quad\text{ on }\lbrace T\rbrace\times\Omega,\\
&z=0\quad\text{ on }(0,T)\times\partial\Omega.
\end{aligned}
\right.
\end{equation}
We can choose $z$ as test function in $(\ref{Eq 7.7})$ and integrate $t$ between $\tau$ and $T$. Since we have from  Young's inequality
\[
\begin{aligned}
\int_{\Omega}\frac{[z(\tau)]^{2}}{2}dx +\frac{1}{2}\int_{\tau}^{T}\int_{\Omega} b'(u)|\nabla z|^{p}dx dt&\leq \int_{\Omega}\frac{[T_{k}(b(u))(T)]^{2}}{2}dx\\
&+C\int_{\tau}^{T}\int_{\Omega}b'(u)|\nabla u|^{p}dx dt
\end{aligned}
\]
we deduce, using also $(\ref{Eq 7.2})$ with $r=T$
\[\int_{\Omega}\frac{[z(\tau)]^{2}}{2}dx +\frac{1}{2}\int_{\tau}^{T}\int_{\Omega}b'(u)|\nabla z|^{p}dx dt\leq Ck(\Vert\mu\Vert_{\mathcal{M}(Q)}+\Vert b(u_{0})\Vert_{L^{1}(\Omega)})= CkM,\]
this implies the estimate for $z$
\begin{equation}\label{Eq 7.8}
\Vert z\Vert_{L^{\infty}(0,T;L^{2}(\Omega))}^{2}+\Vert z\Vert_{L^{p}(0,T;W^{1,p}_{0}(\Omega))}^{p}\leq CkM.
\end{equation}
Since by the definition of $V$ (i.e. $V=W^{1,p}_{0}(\Omega)\cap L^{2}(\Omega)$), we have
\[\Vert z\Vert_{L^{p}(0,T;V)}^{p}\leq C(\Vert z\Vert_{L^{p}(0,T;W^{1,p}_{0}(\Omega))}^{p}+\Vert z\Vert_{L^{p}(0,T;L^{2}(\Omega))}^{p}), \]
Then we have from $(\ref{Eq 7.8})$ that
\begin{equation}\label{Eq 7.9}
\Vert z\Vert_{L^{p}(0,T;V)}\leq C[(kM)^{\frac{1}{p}}+(kM)^{\frac{1}{2}}],
\end{equation}
using the  equation  $(\ref{Eq 7.7})$, we obtain
\[ \Vert z_{t}\Vert_{L^{p'}(0,T;W^{-1,p'}(\Omega))}\leq C(\Vert z\Vert_{L^{p}(0,T;W^{1,p}_{0}(\Omega))}^{p-1}+\Vert T_{k}(b(u))\Vert_{L^{p}(0,T;W^{1,p}_{0}(\Omega))}^{p-1}),\]
hence, we get from $(\ref{Eq 7.5})$ and $(\ref{Eq 7.8})$
\begin{equation}\label{Eq 7.10}
\Vert z\Vert_{L^{p'}(0,T;W^{-1,p'}(\Omega))}\leq C(kM)^{\frac{1}{p'}}.
\end{equation}
Putting together $(\ref{Eq 7.9})$ and $(\ref{Eq 7.10})$, we have the result 
\begin{equation}\label{Eq 7.11}
\Vert z\Vert_{W}\leq C\text{max}\lbrace (kM)^{\frac{1}{p}}, (kM)^{\frac{1}{p'}}\rbrace,
\end{equation}
where $M$ is the constant defined in $(\ref{Eq 7.6})$.\\

\textit{Step.3} Proof completed.\\
Obtaining the energy inequality $(\ref{Eq 7.11})$ was the main step in order to prove the estimate of the capacity $(\ref{Eq 2.6})$. It should be noticed that we  assume that $\mu\geq 0$ to obtain $b(u)_{t}-\Delta_{p}u\geq 0$, $u\geq 0$ in $Q$ and the following inequality holds 
\begin{equation}\label{Eq 7.12}
 (T_{k}(b(u)))_{t}-\Delta_{p}T_{k}(b(u))\geq 0.
\end{equation}
Indeed, one can choose $T'_{k,\eta}(b(u))\varphi$ in $(\ref{Eq 7.1})$ (where $\varphi\in C^{\infty}_{c}(Q)$ and $\varphi\geq 0$), using this time $\mu\geq 0$, with the fact that $T_{k,\eta}(s)$ is concave for $s\geq 0$,
\[ -\int_{0}^{T}\varphi_{t}T_{k,\eta}(b(u))dt+\int_{Q}b'(u)|\nabla u|^{p-2}\nabla u\cdot\nabla\varphi S_{k,\eta}(u)dx dt\geq 0,\]
which yields $(\ref{Eq 7.12})$ as $\eta$ goes to $0$.\\
Therefore, the combinaison of $(\ref{Eq 7.7})$ and $(\ref{Eq 7.12})$ gives
\begin{equation}\label{Eq 7.13}
-z_{t}-\Delta_{p}z\geq -(T_{k}(b(u)))_{t}-\Delta_{p}T_{k}(b(u)).
\end{equation}
We are left to prove that $z\geq T_{k}(b(u))$ a.e. in $Q$ (in particular, $z\geq k$ a.e. on $\lbrace b(u)>k\rbrace$). This is done by means of  $(z-T_{k}(b(u)))^{-}$ in  both sides of $(\ref{Eq 7.13})$, and since $z$ and $T_{k}(u)$ belongs to $L^{p}(0,T;W^{1,p}_{0}(\Omega))$. Indeed we have $u$ has a unique $\text{cap}_{p}-$quasi continuous representative ( recall that, $u$ belongs to $W$); hence, the set $\lbrace b(u)>k\rbrace$ is $\text{cap}-$quasi open, and its capacity can be estimated with  $(\ref{Eq 2.1})$. So that
\[ \text{cap}_{p}(\lbrace |b(u)|>k\rbrace)\leq\Vert\frac{z}{k}\Vert_{W}.\]
Using $(\ref{Eq 7.11})$ and by means that the result is also true for $\mu\leq 0$, we conclude the extension of $(\ref{Eq 2.6})$.
\end{proof}

\end{document}